\documentclass[12pt]{article}
\usepackage{graphicx}
\usepackage{xcolor}
\usepackage{amssymb,amsfonts,amsthm,amsmath,calligra,tikz}
\usepackage[utf8]{inputenc}
\usepackage{accents}
\usepackage{slashed}
\usepackage{yfonts}
\usepackage{mathrsfs,pifont}
\theoremstyle{definition}
\usepackage{tikz}
\usepackage[all]{xy}
\usepackage{enumitem}
\usepackage{float}

\def\be{\begin{eqnarray}}
\def\ee{\end{eqnarray}}

\def\matZ{{\mathbb{Z}}}
\def\matR{{\mathbb{R}}}
\def\matQ{{\mathbb{Q}}}
\def\matC{{\mathbb{C}}}

\def\pr{\mathsf{pr}}

\newcommand{\Stab}{\mathrm{Stab}}
\newcommand{\bA}{\mathsf{A}}

\newcommand{\bT}{\mathsf{T}}
\newcommand{\bK}{\mathsf{K}}

\newcommand{\Lie}{\mathrm{Lie}}
\newcommand{\Pic}{\mathrm{Pic}}

\newcommand{\cL}{\mathscr{L}}

\let\bs\boldsymbol

\def\zz{{\bs z}}

\def\nn{{\bs \nu}}

\def\aa{{\bs a}}

\def\sl{{\bs s}}

\def\wall{\mathsf{w}}

\def\ind{\mathrm{ind}}
\def\Ell{\mathrm{Ell}}

\theoremstyle{definition}
\newtheorem{Definition}{Definition}
\newtheorem{Proposition}{Proposition}
\newtheorem{Lemma}{Lemma}
\newtheorem{Corollary}{Corollary}

\newtheorem{Theorem}{Theorem}

\newtheorem{Remark}{Remark}
\newtheorem{Example}{Example}

\newcommand{\fC}{\mathfrak{C}} 
\newcommand{\fD}{\mathfrak{D}} 

\def\X{\mathsf{X}} 
\def\Y{\mathsf{Y}} 
\def\ms{\Phi} 
\def\mf{\frak{m}} 
\def\hh{\mathsf{H}} 

\makeatletter
\newcommand{\somespecialrotate}[3][]{%
\begingroup
\sbox\@tempboxa{#3}%
\@tempdima=.5\wd\@tempboxa
\sbox\@tempboxa{\rotatebox[#1]{#2}{\usebox\@tempboxa}}%
\advance\@tempdima by -.5\wd\@tempboxa
\mbox{\hskip\@tempdima\usebox\@tempboxa}%
\endgroup}
\linethickness{1 pt}

\def\zz{{\bs z}}

\begin{document}
\title{Pursuing quantum difference equations II:
$3D$-mirror symmetry}
\author{Yakov Kononov, Andrey Smirnov}
\date{}
\maketitle
\thispagestyle{empty}
	
\begin{abstract}
Let  $\X$ and $\X^{!}$ be a pair of symplectic varieties dual with respect to $3D$-mirror symmetry.  The $K$-theoretic limit of the elliptic duality interface is an equivariant $K$-theory class $\mf \in K(\X\times \X^{!})$. We show that this class provides correspondences
$$
\ms_{\mf} : K(\X) \leftrightarrows K(\X^{!})
$$
mapping the $K$-theoretic stable envelopes to the $K$-theoretic stable envelopes. This construction allows us to extend the action of various representation theoretic objects on $K(\X)$, such as action of quantum groups, quantum Weyl groups, $R$-matrices etc., to their action on $K(\X^{!})$. In particular, we relate the wall $R$-matrices of $\X$ to the $R$-matrices of the dual variety $\X^{!}$. 

As an example, we apply our results to  $\X=\mathrm{Hilb}^{n}(\matC^2)$ -- the Hilbert scheme  of $n$ points in the complex plane. In this case we arrive at
the conjectures of E.Gorsky and A.Negut  from \cite{NegGor}.
\end{abstract}

\setcounter{tocdepth}{2}
	
	
\section{Introduction}	

\subsection{}


A class of quantum field theories known as 
{\it ``$\mathcal{N}=4$ three-dimensional SUSY theories''} has been recently attracting growing attention in theoretical physics. These theories have proven to be very rich in properties and, more importantly for our purposes, they are intimately connected with geometric representation theory.

The low energy behaviour of $3D$-theories is governed by the moduli spaces of vacua which, from the mathematical standpoint, are certain singular symplectic varieties. The examples include Nakajima quiver varieties,  slices in affine Grassmannians, Hilbert schemes of points and moduli spaces of sheaves on surfaces - the objects of central importance in contemporaty geometric representation theory.

An important feature of $3D$-theories is the existence of a duality called {\it $3D$-mirror symmetry}. Informally speaking, for a $3D$-theory $\mathcal{T}$ this duality assigns a ``mirror'' theory $\mathcal{T}^{\,!}$ which has the same correlation functions as~$\mathcal{T}$. One could say that $\mathcal{T}$ and $\mathcal{T}^{\,!}$ are two ``languages'' describing the same phenomena.  

In the low energy approximation, the $3D$-mirror symmetry relates the corresponding vacua moduli spaces:
$$
3D\textrm{\it -mirror symmetry} : \ \ \X \leftrightarrows \X^{!}
$$
It is expected that enumerative and topological invariants of the symplectic varieties  $\X$ and $\X^{!}$ are connected in a nontrivial way. In this paper we study these connections  at the level of  equivariant elliptic cohomology and  $K$-theory.

\subsection{}
 
In algebraic geometry $3D$-mirror symmetry is governed by a certain class $\mf$ in the equivariant elliptic cohomology of $\,\X\times \X^{!}$,  which is known as the {\it duality interface}\footnote{This class was called ``the mother function'' in \cite{RSVZ1}}. The duality interface provides the kernel for the elliptic version of ``Fourier-Mukai transform'', which maps the enumerative invariants of $\X$ to those of $\X^{!}$.   Schematically, the partition function $\mathcal{Z}$ of a $3D$-theory $\mathcal{T}$ and the partition function $\mathcal{Z}^{!}$ of the dual $3D$-theory are related by the Fourier transform 
\be \label{duality}
\mathcal{Z}^{!} = \int\limits_{\X} \, \mf\, \mathcal{Z}
\ee
The partition functions in this case can be defined in a mathematically rigorous way as the generating functions for certain equivariant count of rational curves in $\X$ known as {\it vertex functions}. We refer to \cite{OkTalk,AOElliptic} for the definition of vertex functions and for precise meaning of (\ref{duality}).

\subsection{} 


In this paper we are mainly concerned with the $K$-theory limit of the duality interface. In this limit, the duality interface degenerates to an equivariant $K$-theory class
$$
\mf_{0} \in K(\X\times \X^{!})
$$
which provides correspondences 
\be \label{ms}
\ms_{\mf_0}: K(\X) \rightarrow K(\X^{!}), \ \ \     \ms^{t}_{\mf_0}: K(\X^{!}) \rightarrow K(\X)
\ee
defined by 
$$
\ms_{\mf_{0}}: c \mapsto    \pr_{\X^{!},*}(\mf_{0}\otimes  \pr_{\X}^{*}(c)), \ \ \ \ms_{\mf_0}^{t}: c \mapsto    \pr_{\X,*}(\mf_{0}\otimes  \pr_{\X^{!}}^{*}(c)),
$$
where $\pr_{\X}$ and $\pr_{\X^{!}}$ denote the canonical projectors 
\be \label{canprog}
\X \xleftarrow{ \ \ \pr_{\X} \ \ }\X\times \X^{!} \xrightarrow{\ \ \pr_{\X^{!}}\ \ }  \X^{!}.
\ee
We show that the correspondences $\ms_{\mf_{0}}$ and $\ms_{\mf_0}^{t}$ {\textit{map the $K$-theoretic stable envelope classes of $\X$ to the $K$-theoretic stable envelope classes of $\X^{!}$}} and vice versa. We will refer to this property as {\it factorization} of $\ms_{\mf_0}$.  

We recall that the $K$-theoretic stable envelope classes of a variety $\X$ is a certain distinguished basis in equivariant $K$-theory $K(\X)$, see \cite{pcmilect,OS} for definitions.

\subsection{} 

As we discuss in the next section, a part of the $3D$-mirror symmetry data is the identification  of the torus fixed points:
\be \label{commonfp}
\textsf{FP}:=\X^{\bT}\cong (\X^{!})^{\bT^{!}}
\ee
where $\bT$ and $\bT^{!}$ denote algebraic tori acting on the dual varieties.
This identification arises when the set $\textsf{FP}$ is finite,
which is the case considered in this paper.

The $K$-theoretic stable envelopes of $\X$ and $\X^{!}$ can be defined as certain classes~\cite{pcmilect}:
$$
\Stab^{\X} \in K(\textsf{FP} \times \X ), \ \ \ \Stab^{\X^{!}} \in K(\textsf{FP} \times \X^{!}).
$$
which may also be viewed as correspondences:
$$
 K(\X) \xleftarrow{ \ \ \Stab^{\X} \ \ } K(\textsf{FP}) \xrightarrow{ \ \ \Stab^{\X^{!}} \ \ } K(\X^{!}).
$$
The factorization of $\ms_{\mf_0}$ then means that
\be \label{factor}
\ms_{\mf_0}=\Stab^{\X^{!}} \circ \Big( \Stab^{\X} \Big)^{t}, \ \ \
\ms^{t}_{\mf_0}=\Stab^{\X}  \circ \Big( \Stab^{\X^{!}} \Big)^{t}.
\ee
which explains the terminology.

\subsection{} 
More generally, we will also consider twisted  $K$-theoretic limits of the duality interface. To an element $\sl \in H^{2}(\X,\matR)$ we will associate a cyclic group $\mu_{\sl}$, acting on $\X^{!}$, and a subvariety $\Y_{\sl} =(\X^{!})^{\mu_{\sl}} \subset \X^{!}$. The twisted $K$-theoretic limit of the duality interface then gives an equivariant $K$-theory class
$$
\frak{m}_{\sl} \in K(\X\times \Y_{\sl}).
$$
Similarly to (\ref{ms}) this class provides correspondences 
\be \label{{K}dual}
\ms_{\mf_\sl}: K_{\bT}(\X) \leftrightarrows K_{\bT^{!}}(\Y_{\sl}),
\ee 
{\it mapping the $K$-theoretic stable envelopes to the $K$-theoretic stable envelopes.}

This construction allows us to extend various representation theoretic objects acting on $K(\X)$, such as $R$-matrices, quantum Weyl and braid groups,  to actions on $K(\Y_{\sl})$ and vice versa.  

\subsection{} 
In $K$-theory the stable envelope bases are determined by a choice of $\sl \in H^{2}(\X,\matR)$ which is called {\it the slope parameter}. It is known that the stable bases change only when $\sl$ crosses hyperplanes from a certain hyperplane arrangement $\textsf{Walls}(\X)\subset H^{2}(\X,\matR)$. The transition matrix $\mathsf{R}^{\X}(\sl)$ between the stable envelope bases on two sides of a wall  is called  {\it the wall $R$-matrix}. 
In Section~\ref{wcsec}, we use $\ms_{\mf_\sl}$ to relate the wall $R$-matrices of $\X$ with those of~$\Y_{\sl}$. We show that, up to a conjugation by a certain explicit diagonal operator, in the basis of common fixed points (\ref{commonfp}) we have an identity
\be \label{rmatrel}
\mathsf{R}^{\X}(\sl) = \prod\limits_{\sl'}  \mathsf{R}^{\Y_{\sl}}(\sl')
\ee
where the right side is the product of the wall $R$-matrices of $\Y_{\sl}$ corresponding to hyperplanes passing through $0\in H^{2}(\Y_{\sl},\matR)$. In other words, the right side is the transition matrix between the stable bases of $K(\Y_{\sl})$ with slopes from ample 
and anti-ample alcoves.

\subsection{} 

The equivariant $K$-theories of varieties appearing as moduli spaces of vacua, are often equipped with natural actions of affine quantum groups
$$
\mathscr{U}_{\hbar}(\widehat{\frak{g}}_{\X}) \curvearrowright\, K(\X).
$$
The theory of the $K$-theoretic stable envelopes is a natural tool to construct and describe this action \cite{pcmilect,OS}.
For instance,   $K$-theoretic stable envelopes  often provide distinguished bases of $K(\X)$ in which the action of quantum groups {\it takes the simplest form}. As an example - the standard and costandard bases of $\mathscr{U}_{\hbar}(\widehat{\frak{gl}}_{n})$-modules  are the incarnations of the stable envelope bases~\cite{Neg}. In this light, the relation between the $K$-theoretic stable envelopes arising from $3D$-mirror symmetry (\ref{{K}dual}) bridges the representation theories of the quantum groups associated with $\X$ and $\Y_{\sl}$. 

In Section \ref{hilbex} we apply this idea to the example $\X=\mathrm{Hilb}^{n}(\matC^2)$ - the Hilbert scheme of $n$ points in the plane. In this case $H^{2}(\X,\matR)\cong \matR$ and interesting phenomena occur at rational points $\sl=\frac{a}{b} \in \matQ$ with $b\leq n$.
The components of
$\Y_{\sl}$ are isomorphic to  Nakajima varieties associated with the cyclic quiver with $b$ vertices, see Fig.\ref{cycl}. The $K$-theories of $\Y_{\sl}$ are  equipped with a natural action of the quantum affine algebra $\mathscr{U}_{\hbar}(\widehat{\frak{gl}}_{b})$. As a   $\mathscr{U}_{\hbar}(\widehat{\frak{gl}}_{b})$-module $K(\Y_{\sl})$ is isomorphic to the so-called Fock representation of level $1$:
$$
 \mathscr{U}_{\hbar}(\widehat{\frak{gl}}_{b}) \curvearrowright\, K(\Y_{\sl}) \cong \mathsf{Fock}
$$
Thus,  $3D$-mirror symmetry  (\ref{{K}dual}) for $\sl=\frac{a}{b}$ provides a certain natural actions of 
$\mathscr{U}_{\hbar}(\widehat{\frak{gl}}_{b})$ on $K$-theories of Hilbert schemes $\mathrm{Hilb}^{n}(\matC^2)$. 

The identity (\ref{rmatrel}) then says that  the transition matrix between the {\it standard} and {\it costandard} bases of the Fock  $\mathscr{U}_{\hbar}(\widehat{\frak{gl}}_{b})$-module is equal to the wall $R$-matrix of $\mathrm{Hilb}^{n}(\matC^2)$ for the wall $\sl=\frac{a}{b}$. In this way we prove the conjectures discussed by E.Gorsky and A.Negut in \cite{NegGor}.

\subsection{} 
The main objective of this series of papers is to use $\ms_{\mf_\sl}$ to gain a better geometric understanding of the {\it wall crossing operators} and the {\it quantum difference equations} discussed in \cite{OS}.  The wall crossing operators, acting on $K(\X)$, provide a geometric version of quantum dynamical Weyl groups~\cite{EV}. 
We expect that $3D$-mirror symmetry exchanges the action of the wall crossing operators on  $K(\X)$ with action of the  dynamical $R$-matrices on  $K(\X^{!})$, which leads us to much deeper understanding of these operators.

Ideas developed here also have direct applications to enumerative geometry, in particular, various limits of {\it vertex functions} investigated in \cite{Dink1,Dink2,dinksmir3,Liu1} is similar to the limits of $\mf$ studied in this paper. 

\section*{Acknowledgements}

\indent
We thank our teacher Andrei Okounkov for introducing us to the
subject discussed in this paper and sharing his ideas.

We are indebted to Andrei Negut for explaining the conjectures of \cite{NegGor}, and helping us to bridge them with  $3D$-mirror symmetry and Eugene Gorsky for reading the preliminary version of the paper and helpful suggestions.

We also thank Noah Arbesfeld, Ivan Cherednik, Ivan Danilenko, Hunter Dinkins,  Boris Feigin, Roman Gonin, Henry Liu, Richard Rimanyi, Alexander Varchenko and Zijun Zhou for discussions and collaborations.

The ideas we develop in this series of papers came from discussions with our colleagues and friends  during the AMS Mathematics Research
Community meeting on Geometric Representation Theory and Equivariant
Elliptic Cohomology at Rhode Island in June 2019 and the workshop “Elliptic
cohomology days” at the University of Illinois, Urbana-Champaign. We are indebted to the organizers and all participants for very fruitful
conversations and creative scientific atmosphere.

A. Smirnov is supported by
the Russian Science Foundation under grant 19-11-00062 and is performed at Steklov Mathematical Institute of Russian Academy of
Sciences.

\section{Data of $3D$-mirror symmetry \label{datasec}}

\subsection{}

Let $\X$ denote a branch of a vacua moduli space of $3D$-theory and $\X^{!}$ the corresponding branch of the dual theory provided by $3D$-mirror symmetry.  As mentioned already, we assume that both $\X$ and $\X^{!}$ are smooth, i.e., we are working with resolutions of singularities.  The resolutions of singularities are typically controlled by a choice of an element $\theta\in H^{2}(\X,\matR)$. Depending on the context, $\theta$ may manifest itself as a choice of a stability condition for quiver varieties, or as a choice of a convolution diagram for slices in affine Grassmannians, etc. 

We denote by $\bT$ be the maximal torus of the group $\textrm{Aut}(\X)$ acting on $\X$ by automorphisms. This torus scales the symplectic form by the character which is traditionally denoted by $\hbar$. We denote by $\bA=\ker(\hbar)$ the subtorus preserving the symplectic form, so that 
$\bT=\bA\times \matC^{\times}_{\hbar}$.
We define the {\it K\"ahler torus} of $\X$ by $\bK=\Pic(\X)\otimes_{\matZ} \matC^{\times}$. It will be convenient to think of $\theta$ as a cocharacter $\theta\in \Lie_{\matR}(\bK)$.


The coordinates in the torus $\bA$ are traditionally referred to as {\it equivariant parameters} and in the torus $\bK$ 
as {\it K\"ahler parameters} due to their role in enumerative geometry. 

We will denote by $\bT^{!},\bA^{!},\bK^{!},\hbar^{!},\theta^{!}$
 the same data associated with $\X^{!}$. 
 
 Given a  torus $\bT$ we will denote by $\bT^{\wedge}$ and $\bT^{\vee}$ the lattices of characters and cocharacters respectively.

 \subsection{}
 
 Another assumption we impose on $\X$ (and $\X^{!}$) is that the set $\X^{\bA}$ is finite. We expect that both the smoothness of $\X$ and finiteness of $\X^{\bA}$ are superfluous, but working in a more general setting requires overcoming significant technical hurdles, see \cite{OkounkovInductive} for current progress in this direction, and we postpone it for further investigations. Even with the mentioned assumptions, the class of varieties for which our treatment applies is large enough. The examples include quiver varieties of finite and affine $A_n$-type, resolutions of slices in affine Grassmannians of $A_n$-type, bow varieties \cite{NakajimaBows}, etc.

\subsection{} 

$3D$-mirror symmetry imposes constraints on the data associated with  dual varieties. The {\it first condition} is the existence of isomoprhisms\footnote{More generally, $\kappa$ also includes an extra torus $\matC^{\times}_q$ acting on the source of quasimaps. For the purposes of this paper this is not relevant, as the elliptic functions we deal with are $q$-periodic.} 
\be \label{kap}
\kappa : \bA \to \bK^{!}, \ \ \ \bK \to \bA^{!}, \ \ \matC^{\times}_\hbar \to \matC^{\times}_{\hbar^{!}}.
\ee

 Informally, (\ref{kap}) means that $3D$-mirror symmetry {\it exchanges the equivariant and the K\"ahler parameters}.

We denote
\be \label{chars}
\sigma = d\kappa^{-1}(\theta^{!}) \in \textrm{Lie}_{\matR}(\bA), \ \ \sigma^{!} =d\kappa(\theta) \in \textrm{Lie}_{\matR}(\bA^{!}).
\ee
The cocharacters $\sigma,\sigma^{!}$ conveniently define attracting and repelling directions in tangent spaces at fixed points. We define the attracting set of $p\in \X^{\bA}$ by
$$
\textrm{Attr}_{\sigma}(p)=\{x \in \X:  \lim\limits_{z\to 0} \sigma(z) \cdot x =p \}
$$
and the full attracting set $\textrm{Attr}^{f}_{\sigma}(p)$ as the smallest closed subset of $\X$ which contains $p$ and is closed under taking $\textrm{Attr}$.  There is a partial ordering on the torus fixed points of $\X$  defined by
\be \label{order}
p_1\succ p_2 \ \  \Leftrightarrow  \ \ p_2 \in \textrm{Attr}_{\sigma}^{f}(p_1),
\ee
The same applies to $\X^{!}$ with $\sigma$ replaced by $\sigma^{!}$.

\subsection{} 
The {\it second condition} is the existence of a bijection 
\be \label{biject}
  \X^{\bA} \to (\X^{!})^{\bA^{!}}
\ee
which {\it inverses}  the partial order on the fixed points. Given an $\bA$-fixed point $p\in \X^{\bA}$ we will 
denote by $p^{!}$ the corresponding $\bA^{!}$ - fixed point of $\X^{!}$.

\subsection{} 
Let $E=\matC^{\times}/q^{\matZ}$ be an elliptic curve and let $\Ell_{\bT}(\X)$ denote the corresponding equivariant elliptic cohomology scheme of a $\bT$-variety $\X$ \cite{AOElliptic,ell3,ell5,ell6,ell1}. For instance
$$
\Ell_{\bT}(pt):=\bT\left/q^{\bT^{\vee}}\right.\cong E^{\dim(\bT)}.
$$
We denote by $\mathsf{E}_{\bT}(\X) = \Ell_{\bT}(\X) \times \textrm{Ell}_{\bT^{!}}(pt)=\Ell_{\bT}(\X) \times (\bK/q^{\bK^{\vee}})$ the extension of this scheme. 

The elliptic stable envelope $\Stab^{\X,Ell}_{\sigma}(p)$ is a section of a certain line bundle over scheme $\mathsf{E}_{\bT}(\X)$ which can be constructed from a choice of $p\in \X^{\bA}$ and {\it generic} $\sigma \in \Lie_{\matR}(\bA)$.  For the definition and construction of $\Stab^{\X,Ell}_{\sigma}(p)$ we refer to \cite{AOElliptic,OkounkovInductive}. 

We recall also that ``generic'' $\sigma$  means that it is from the set
\be \label{equivcham}
\Lie_{\matR}(\bA) \setminus \{ \sigma:  \langle v , \sigma \rangle =0, \ \ v \in \textrm{char}_{\bA}(T_p \X), \  \   p\in \X^{\bA}   \} =\coprod \fC
\ee
where $\fC$ denote chambers representing connected components of this set.

The stable envelope $\Stab^{\X,Ell}_{\sigma}(p)$, as a function of $\sigma$, depends only on the chamber $\fC$. 


\subsection{}

For a $\bT$-fixed point $p\in \X$, and $\bT^{!}$-fixed point $p^{!}\in \X^{!}$ we have $\bT \times \bT^{!}$-equivariant embeddings:
$$
\X = \X\times \{ p^{!} \} \longrightarrow \X\times \X^{!} \longleftarrow  \{p\}\times \X^{!} =\X^{!}.
$$
Functoriality of elliptic cohomology induces:
$$
\mathsf{E}_{\bT}(\X) \stackrel{i_{p^{!}*}}{\longrightarrow} \mathsf{Ell}_{\bT\times \bT^{!}}(\X\times \X^{!}) \stackrel{ \ \ i_{p*}}{\longleftarrow} \mathsf{E}_{\bT^{!}}(\X^{!}).
$$
where $\mathsf{E}_{\bT}(\X) = \Ell_{\bT}(\X) \times \textrm{Ell}_{\bT^{!}}(pt)=\Ell_{\bT}(\X) \times (\bK/q^{\bK^{\vee}})$, see \cite{RSVZ1,RSVZ2} for details of this construction.

We will need a twisted version of this section, which differs from normalization accepted in \cite{AOElliptic} by a prefactor:
$$
\pmb{\Stab}^{\X,Ell}_{\sigma}(p)=\Theta(N^{-}_{p^!})\cdot  \Stab^{\X,Ell}_{\sigma}(p)
$$
where $\Theta(N^{-}_{p^!})$ is the section given explicitly by 
$$
\Theta(N^{-}_{p^!})=\prod\limits_{{w \in \textrm{char}_{\bT^{!}}( T_{p^!} \X^{!})} \atop {\sigma^{!}(w)<0}} \vartheta(\aa^{w})
$$
i.e., the product goes over the $\bT^{!}$-characters of the tangent space  which take negative values at the cocharacter $\sigma^{!}$.  Similarly, we have twisted stable envelopes on the dual side
$$
\pmb{\Stab}^{\X^{!},Ell}_{\sigma^{!}}(p^{!})=\Theta(N^{-}_{p})\cdot  \Stab^{\X^{!},Ell}_{\sigma^{!}}(p^{!}).
$$
From the definition of the stable envelopes \cite{AOElliptic} it follows that in this normalization we have: 
\be \label{norm}
\left.\pmb{\Stab}^{\X,Ell}_{\sigma}(p)\right|_{p}\left.=\pmb{\Stab}^{\X^{!},Ell}_{\sigma^{!}}(p^{!})\right|_{p^!}= \Theta(N^{-}_{p}) \cdot  \Theta(N^{-}_{p^{!}})
\ee
The {\it third condition} imposed by $3D$-mirror symmetry is the requirement that the sections $\pmb{\Stab}^{\X,Ell}_{\sigma}(p)$ , $p\in \X^{\bA}$ and $\pmb{\Stab}^{\X^{!},Ell}_{\sigma^{!}}(p^{!})$, $p^{!} \in (\X^{!})^{\bA^{!}}$ {\it glue to one global section over elliptic cohomology of} $\X\times \X^{!}$:

\vspace{2mm}

\begin{Definition} \label{mirdef} {\it
We say that a variety $\X^{!}$ is a $3D$-mirror of $\X$ if:
\begin{itemize}
\item There is an isomorphism (\ref{kap}),
\item There is a bijection (\ref{biject}), 
\item There exists a line bundle $\mathfrak{M}$ over the  scheme $\mathrm{Ell}_{\bT\times \bT^{!}}(\X\times \X^{!})$ and a section $\mathfrak{m}$ such that:
\be \label{mird}
i^{*}_{p^!}(\mathfrak{m})= \pmb{\Stab}^{\X,Ell}_{\sigma}(p), \ \ \ i^{*}_{p}(\mathfrak{m})= \pmb{\Stab}^{\X^{!},Ell}_{\sigma^{!}}(p^{!}).
\ee
\end{itemize}}
\end{Definition}


\subsection{} 
Many examples of pairs $\X$ and $\X^{!}$ satisfying Definition \ref{mirdef} have been found recently. In \cite{RSVZ1} for $\X=T^{*}Gr(k,n)$ with $n\geq 2k$ (where $Gr(k,n)$ denotes the Grassmannian of $k$-planes in $\matC^n$) we construct $\X^{!}$ as a certain quiver variety. In \cite{RSVZ2} we show that $\X\cong \X^{!} \cong T^{*}(G/B)$  where $G/B$ stands for the full flag varieties of $G=GL(n)$. This result were further extended to  flag varieties of arbitrary Lie groups in \cite{RimanWeb}, in which case $\X=T^{*}(G/B)$  and $\X^{!} = T^{*}(G^{L}/B^{L})$ where $G^{L}$ denotes the Langlands dual of $G$.

Finally, in \cite{SmirnovZhuHypertoric}  $\X^{!}$ is constructed for an arbitrary hypertoric variety $\X$. In this case, the duality interface $\mathfrak{m}$ can be described explicitly, see Theorem 6.4 in \cite{SmirnovZhuHypertoric}.

In general, if one has a vacua moduli space $\X$, physicists predict what $\X^{!}$ is. For instance, if $\X$ is a {\it bow variety} then $\X^{!}$ is  the bow variety obtained by switching the 
$\bullet$ - type and $\textsf{x}$ - type vertices in the bow diagram of $\X$ \cite{NakajimaBows}. Computing the stable envelopes and checking properties listed in Definition \ref{mirdef} is, however, a non-trivial problem. Nevertheless, we expect that the list of pairs $(\X,\X^{!})$ satisfying Definition \ref{mirdef} will grow in the nearest future.

\subsection{} 
In a very general setting, one can define $3D$-theory for a pair $(G,M)$ where $G$ is a simply connected Lie group and $M$ its symplectic representation. In this case one defines the ``Higgs branch'' of the theory as a hyperk\"ahler reduction:
$$
\X=M/\!\!/\!\!/\!\!/ G
$$
Recently, in the series of papers \cite{CoulombBoranches1,CoulombBoranches2,CoulombBranches} the authors proposed a mathematical definition of the ``Coulomb branch'' of a $3D$-theory.  It is expected that if  the Coulomb branch admits symplectic resolution with finite set of fixed points then this construction provides $\X^{!}$.   The examples discussed in the previous subsection are in agreement with this expectation. 

\section{Quasiperiods of the restriction matrices}

\subsection{} 
We identify characters and cocharacters of the K\"ahler torus $\bK$ with
$$
\bK^{\vee}=H^{2}(X,\matZ), \ \ \ \bK^{\wedge}=H_{2}(X,\matZ).
$$
in particular $\Lie_{\matR}(\bK)=H^{2}(X,\matR)$. We assume that 
$$
c_1: \Pic(X) \rightarrow H^{2}(X,\matZ)
$$
defined by the first Chern class $\cL \to c_1(\cL)$ is an isomorphism of lattices.

\subsection{}
For a fixed point $p\in \X^{\bA}$ we have a natural homomorphism:
$$
\chi_{p} : \bK^{\vee} \to \bA^{\wedge}, \ \ \ c_1(\cL) \mapsto  \delta, \ \ \textrm{determined by} \ \ \aa^{\delta}=\left.\cL\right|_{p} \in K_{\bA}(pt).  
$$
Depending on the context, it will be convenient to view it as a pairing 
$$
\chi_{p} : \bK^{\vee}\otimes \bA^{\vee} \to \matZ. 
$$
or  as a map
$$
\chi : \X^{\bA} \rightarrow \bK^{\wedge}\otimes \bA^{\wedge}.
$$

\begin{Proposition} \label{prodif}
{\it Let $C\cong \mathbb{P}^{1}$ be a $\bA$-equivariant curve in $\X$ connecting two torus fixed point $p_{+},p_{-}\in \X^{\bA}$. Let $\sigma \in \bA^{\wedge}$ be a generic cocharacter. Assume that $v \in \bA^{\wedge}$ is the character of
$T_{p_{+}} C$ (then the character of $T_{p_{-}} C$ equals $-v$) such that $\langle v,\sigma \rangle >0$, then
\be \label{chipropp}
\chi_{p_{+}} - \chi_{p_{-}} = [C] \otimes v
\ee}
\end{Proposition}
\begin{proof} 
By the $\bA$-equivariant localization
$$
\langle c_{1}(\cL),[C] \rangle= \dfrac{\left.c_{1}(\cL)\right|_{p_{+}}}{v}+ \dfrac{\left.c_{1}(\cL)\right|_{p_{-}}}{-v},
$$
thus
$$
\left.c_{1}(\cL)\right|_{p_{+}}-\left.c_{1}(\cL)\right|_{p_{+}}=\langle c_{1}(\cL),[C] \rangle v
$$
which is exactly the value of $\chi_{p_{+}}-\chi_{p_{-}}$ at $c_{1}(\cL)$.
\end{proof}

\begin{Remark}
$(\ref{chipropp})$ defines $\chi$ up to a shift by an element of $\bK^{\wedge}\otimes \bA^{\wedge}$. In particular, it determines $\chi$ from its value at one point $p\in \X^{\bA}$. 
\end{Remark}

\subsection{} 
Let us consider the matrix:
\be \label{tnorm}
\tilde{T}^{\X}_{p,r}(\zz,\aa):= \dfrac{\left.\Stab^{\X,Ell}_{\sigma}(p)\right|_{r}  }{\Theta(N^{-}_{r})}=\dfrac{\left.\pmb{\Stab}^{\X,Ell}_{\sigma}(p)\right|_{r}  }{\Theta(N^{-}_{r}) \Theta(N^{-}_{p^{!}})  }, \ \ \ p,r \in \X^{\bA},
\ee
consisting of the fixed point components of elliptic stable envelopes. By~(\ref{norm}) $\tilde{T}^{\X}_{r,r}(\zz,\aa)=1$.
 The identity (\ref{mird}) then gives:
\be \label{tmir}
\tilde{T}^{\X}_{p,r}(\zz,\aa)=\kappa^{*}(\tilde{T}^{\X^{!}}_{r^{!},p^{!}}(\zz,\aa)).
\ee
where $\kappa^{*}$ means that the parameters of $\X$ are identified with those of $\X^{!}$ via~$\kappa$. 

\begin{Remark}
By definition, $\Stab^{\X,Ell}_{\sigma}(p)$ is supported at $\textrm{Attr}^{f}_{\sigma}(p)$, and thus 
$\tilde{T}^{X}_{p,r}(\zz,\aa)$ is a {\it lower triangular matrix} if the fixed points $\X^{\bA}$
are ordered by~(\ref{order}) (from lowest to highest) associated with $\sigma$. 
\end{Remark}

\begin{Remark}
From (\ref{tmir}) we also see the partial orders on $\X^{\bA}$ and $(\X)^{\bA^{!}}$ associated with $\sigma$ and $\sigma^{!}$ must be inverses of each other.
\end{Remark}


\subsection{} 
The map $\chi$ controls quasiperiods of the elliptic stable envelopes \cite{AOElliptic}. For $\sigma \in \bA^{\vee}$ and $\delta\in \bK^{\vee}$ we have:
$$
\begin{array}{l}
\tilde{T}^{\X}_{p,r}(\zz q^{\sigma},\aa)= \aa^{\chi_{p}(\sigma,\cdot)-\chi_{r}(\sigma,\cdot)} \tilde{T}^{\X}_{p,r}(\zz,\aa), \\
\\
\tilde{T}^{\X}_{p,r}(\zz ,\aa q^{\delta})= \zz^{\chi_{p}(\cdot,\delta)-\chi_{r}(\cdot,\delta)} \tilde{T}^{\X}_{p,r}(\zz,\aa).
\end{array}
$$
It also controls vanishing of the matrix elements $\tilde{T}^{\X}_{p,r}(\zz,\aa)$:
\begin{Proposition} \label{dprop}
{\it Let $\tilde{T}^{\X}_{p,r}(\zz,\aa)$ be the matrix of restrictions of the stable envelopes corresponding to a chamber $\fC \subset \Lie_{\matR}(\bA)$. If $\tilde{T}^{\X}_{p,r}(\zz,\aa)\neq 0$ then 
$$
\chi_{p}-\chi_{r} \in H_{2}(\X,\matZ)_{\mathit{eff}} \otimes \bA^{\wedge}_{>}
$$
where $\bA^{\wedge}_{>}$ denotes the cone of characters positive on $\fC$.}
\end{Proposition}
\begin{proof}
The elliptic stable envelope of $p\in \X^{\bA}$ is supported at $\mathrm{Attr}^{f}(p)$, so
if $\tilde{T}^{\X}_{p,r}(\zz,\aa)\neq 0$ then $r\in \mathrm{Attr}^{f}(p)$. This condition means that there exists a chain of invariant curves $C$ connecting the points $p$ and $r$ such that the character $T_{p} C$ is positive at $\fC$. The result follows from Proposition~\ref{prodif}. 
\end{proof}




\section{K-theoretic limits for regular slopes}
\subsection{} 
Before we discuss the general situation, let us consider an example which reveals relevant properties of the elliptic functions and their trigonometric limits. Let us consider a function 
$$
f(z,a)=\dfrac{\vartheta(a z)}{\vartheta(a) \vartheta(z)}.
$$
where 
$$
\vartheta(x)=(x^{1/2}-x^{-1/2}) \prod\limits_{i=1}^{\infty} (1-x q^{i}) (1-x^{-1} q^{i})
$$
is the odd Jacobi theta-function. For $s\in \matR$ an elementary calculation gives:
$$
\lim 
\limits_{q\to 0} f(z q^{s},a)=\left\{\begin{array}{ll}
\dfrac{a^{-\lfloor s \rfloor}}{a-1}, & s\not\in \matZ,\\
\dfrac{1-z a}{(a-1)(1-z)} a^{-s},& s\in \matZ,
\end{array}\right.
$$
where $\lfloor s \rfloor \in \matZ$ denotes the integral part of $s$. The same works for the limit of $f(z,a q^{s})$ by symmetry $z\leftrightarrow a$. 
\vspace{2mm}

\noindent
We see that the $q\to 0$ limit of $f(z q^{s},a)$

\begin{itemize}
\item is a piecewise constant functions of $s\in \matR$,

\item  changes only when $s$ crosses ``walls'' located at $\matZ\subset \matR$, 

\item for regular $s$, i.e., $s\in \matR\setminus\matZ$ the limit does not depend on $z$. 
\end{itemize}

\subsection{}

The $K$-theoretic limit of the elliptic stable envelopes is a multivariable version of the previous example. 

\begin{Theorem}[\cite{AOElliptic,OkounkovInductive}] \label{thm1}
{\it For $\sl \in H^{2}(\X,\matR)$ the limit $\lim\limits_{q\to 0} \tilde{T}_{p,r}(\zz q^{\sl},\aa)$:
\begin{itemize} 
\item is a piecewise constant function of $\sl$,
\item changes only when $\sl$ crosses a hyperplane from a certain hyperplane arrangement $\mathsf{Walls}(\X) \subset H^{2}(\X,\matR)$,
\item for regular slopes  $\sl \in H^{2}(\X,\matR)\setminus \mathsf{Walls}(\X)$ the limit does not depend on the K\"ahler parameters $\zz$. In this case the limit equals
$$
\lim\limits_{q\to 0} \tilde{T}_{p,r}(\zz q^{\sl},\aa)=\tilde{A}^{[\sl],\X}_{p,r}
$$
where $\tilde{A}^{[\sl],\X}_{p,r}$ is the  matrix of fixed point components of the $K$-theoretic stable envelopes of $\X$ with the slope $\sl$:
\be \label{amat}
\tilde{A}^{[\sl],\X}_{p,r} = \dfrac{\left.\Stab^{[\sl],\X,{K}}_{\sigma}(p)\right|_{r}}{\left.\Stab^{[\sl],\X,{K}}_{\sigma}(r)\right|_{r}}.
\ee
\end{itemize}}
\end{Theorem}
\begin{Remark}
We note that the matrix $\tilde{A}^{[\sl],\X}_{p,r}$ is normalized as in (\ref{tnorm}) so that $\tilde{A}^{[\sl],\X}_{r,r}=1$ for $r\in \X^{\bA}$. \end{Remark}
\noindent
For the definitions of the $K$-theoretic stable envelope classes $\Stab^{[\sl],\X,{K}}_{\sigma}(p) \in K_{\bT}(\X)$ of fixed points $p\in\X^{\bT}$ we refer to $\cite{pcmilect,OS}$.

\subsection{ \label{ressect}} 
For an element $\mathsf{w}=({w}_1,\dots,{w}_n) \in \Lie_{\matR}(\bA) \cong \matR^{n}$ we consider $\omega=e^{2\pi i  \mathsf{w}} =(e^{2\pi i {w}_1},\dots, e^{2\pi i {w}_n} ) \in \bA$. Let $\nn_{\mathsf{w}} \subset \bA$ be the cyclic subgroup generated by $\omega$. In \cite{KononovSmirnov1} we  considered the following set
$$
\textsf{Res}({\X}):=\{ \mathsf{w} : \X^{\nn_{\mathsf{w}}}\neq \X^{\bA} \} \subset \Lie_{\matR}(\bA).
$$
which we called the set of {\it resonances}. It is known that $\textsf{Res}({\X})$ is an arrangement of hyperplanes in $\Lie_{\matR}(\bA)$ given explicitly by:
$$
\textsf{Res}({\X})=\{ \wall \in \Lie_{\matR}(\bA): \langle \alpha,\wall \rangle + m=0, \ \  m\in \matZ, \ \ \alpha \in \mathrm{char}_{\bA}(T_{p} \X), \  p \in \X^{\bA}  \}
$$
see Proposition 5 in \cite{KononovSmirnov1}.

\begin{Theorem} \label{thm2}
{\it $3D$-mirror symmetry switches the walls with the resonances:
$$
\mathsf{Res}(\X)=\mathsf{Walls}(\X^{!}), \ \ \ \mathsf{Res}(\X^{!})=\mathsf{Walls}(\X),
$$
where we identify $\Lie_{\matR}(\bA)\cong \Lie_{\matR}(\bK^{!})$ and $\Lie_{\matR}(\bK)\cong \Lie_{\matR}(\bA^{!})$ via $\kappa$.

\vspace{2mm}
\noindent The limit 
$\lim\limits_{q\to 0} \tilde{T}_{p,r}(\zz ,\aa q^{\mathsf{w}})$: 
\begin{itemize}
    \item is a piecewise constant function of $\mathsf{w}\in \Lie_{\matR}(\bA)$,
    \item changes only when $\mathsf{w}$ crosses a hyperplane in $\mathsf{Res}({\X})$,
    \item is independent on the equivariant parameters $\aa$ when $\mathsf{w} \in \Lie_{\matR}(\bA)\setminus \mathsf{Res}({\X})$. 
    
    In this case the limit is equal
    $$
    \lim\limits_{q\to 0} \tilde{T}_{p,r}(\zz ,\aa q^{\mathsf{w}})=\tilde{Z}^{[\mathsf{w}],\X^{!},{K}}_{r^{!}, p^{!}}
    $$
    where $\tilde{Z}^{[\mathsf{w}],\X^{!},{K}}_{r^{!}, p^{!}}$ denotes the matrix of fixed point components of the $K$-theoretic stable envelopes of $\X^{!}$ with slope $\mathsf{w} \in \Lie_{\matR}(\bA)\cong H^{2}(\X^{!},\matR)$:
    \be \label{zmat}
    \tilde{Z}^{[\mathsf{w}],\X^{!}}_{p^{!},r^{!}} = \dfrac{\left.\Stab^{[\mathsf{w}],\X^{!},{K}}_{\sigma^{!}}(p^{!})\right|_{r^{!}}}{\left.\Stab^{[\mathsf{w}],\X^{!},{K}}_{\sigma^{!}}(r^!)\right|_{r^!}}
    \ee
    where we assume that the fixed points and the parameters are identified by (\ref{biject}) and (\ref{kap}).
\end{itemize}
}
\end{Theorem}
\begin{proof}
All statements follow from Theorem \ref{thm1}, Proposition 5 in \cite{KononovSmirnov1} and $3D$-mirror symmetry relation~(\ref{tmir}).
\end{proof}


\section{K-theoretic limit for non-regular slopes}
\subsection{} 

We need the following orthogonality of the $K$-theoretic stable envelopes
\begin{Lemma} \label{orthlem}
$$
\chi_{X}\left(\Stab^{[-s],\X,{K}}_{-\sigma}(p) \otimes \Stab^{[s],\X,{K}}_{\sigma}(r)\right) =\delta_{p,r}
$$
\end{Lemma}
\begin{proof}
Proposition 1 in \cite{OS}.
\end{proof}

\noindent
Recall that a slope $\sl \in H^{2}(\X,\matR)\setminus \mathsf{Walls}(\X)$ is called {\it regular}. 

The following theorem described the limit of the elliptic stable envelope to a wall as a product of two operators, one of which depends significantly on equivariant, and the other on K\"ahler parameters.
\begin{Theorem} \label{factortheorem}
{\it Let $\sl \in H^{2}(\X,\matR)$ and $\sl'$ is a regular slope from a small analytic neighborhood of $\sl$, then the limit factorizes
$$
\lim\limits_{q\to0}\tilde{T}(\zz q^{\sl},\aa)= \tilde{Z}^{''} \tilde{A}^{[\sl'],\X} 
$$
where $\tilde{A}^{[\sl'],\X}=(\tilde{A}^{[\sl'],\X}_{p,r})_{p,r\in \X^{\bA}}$ is the matrix of $K$-theoretic stable envelope of $\X$ with slope $\sl'$ defined by (\ref{amat}). The matrix elements of $\tilde{Z}$ are monomials in~$\aa$:
$$
\tilde{Z}_{p,r}^{''}= Z_{p,r}^{'} \aa^{\chi_{p}(\sl,\cdot,)-\chi_{r}(\sl,\cdot)}, \ \ \tilde{Z}_{p,r}^{'} \in \matQ(\zz,\hbar). 
$$
In particular 
$$
Z^{'}_{p,r} \neq 0 \ \ \Rightarrow \ \ \chi_{p}(\sl,\cdot,)-\chi_{r}(\sl,\cdot)\in \bA^{\wedge}.
$$}
\end{Theorem}

\begin{proof}
Let $T_{p,r}(\zz,\aa):= \left.\Stab^{\X,Ell}_{\sigma}(p)\right|_{r}$.
The collection $\lim\limits_{q\to 0} {T}_{p,r}(\zz\, q^{\sl},\aa)$ for $r\in \X^{\bA}$  are the fixed point components of certain integral $K$-theory class, which we denote by 
$$
\Gamma(p) \in K_{\bT}(X) \otimes \matQ(\zz,\hbar).
$$
By Theorem \ref{thm1} we have
$$
 \lim\limits_{q\to 0} {T}_{p,r}(\zz\, q^{\sl'},\aa)=A^{[\sl'],\X}_{p,r}=\left.\Stab^{[\sl'],\X,{K}}_{\sigma}(p)\right|_{r}
$$
Let us consider the matrix
$$
\tilde{Z}'':=\left(\lim\limits_{q\to 0} \tilde{T}(\zz\, q^{\sl},\aa)\right) \left(\tilde{A}^{[\sl'],\X}\right)^{-1} =\left(\lim\limits_{q\to 0} {T}(\zz\, q^{\sl},\aa)\right) \left({A}^{[\sl'],\X}\right)^{-1}
$$
By Lemma~\ref{orthlem}, the inverse of the matrix $A^{[\sl'],\X}$ is the matrix of $K$-theoretic stable envelope for inverse cocharacter $-\sigma$ and inverse slope $-\sl'$. This means that 
$$
\tilde{Z}^{''}_{p,r} = \chi_{\X}\Big( \Stab^{[-\sl'],\X,{K}}_{-\sigma}(r) \otimes\Gamma(p) \Big)  \in K_{\bT}(pt) \otimes \matQ(\zz,\hbar).
$$
In particular, $\tilde{Z}^{''}_{p,r}$ is a Laurent polynomial in equivariant parameters $\aa$. Computation in the equivariant localization gives
\be \label{rsum}
\tilde{Z}^{''}_{p,r} = \sum\limits_{l\in \X^{\bT}} \, \dfrac{ \left.\Stab^{[-\sl'],\X,{K}}_{-\sigma}(r)\right|_{l} \left.\Gamma(p)\right|_{l}}{\bigwedge^{\!\bullet}{(T_l \X^{\vee})}}
\ee
The $\aa$-degrees of terms in this sum can be estimates from 
the ``window condition'' for stable envelopes, see Section 9.1.9 in \cite{pcmilect}. This condition gives:

\be \label{np1}
\deg_{\bA}\left( \left.\Stab^{[-\sl'],\X,{K}}_{-\sigma}(r)\right|_{l} \right) \subset \mathrm{NP}^{+} + \chi_{r}(-\sl',\cdot)-\chi_{l}(-\sl',\cdot).
\ee
and
\be \label{np2}
\deg_{\bA}\left( \left.\Gamma(p)\right|_{l} \right) \subseteq \mathrm{NP}^{-} + \chi_{p}(\sl,\cdot)-\chi_{l}(\sl,\cdot),
\ee
where
$$
\mathrm{NP}^{\pm} = \textrm{Newton polygon of the Laurent polynomial} \ \  \prod\limits_{{\delta \in\mathrm{char}_{\bA}(T_l \X)}\atop {\pm \sigma(\delta)>0}} \left(1-\aa^{\delta}\right).
$$
We recall that (\ref{np2}) means that the the $\aa$-Newton polygon of a Laurent polynomial $\left.\Gamma(p)\right|_{l}$ is contained in the polygon $\mathrm{NP}^{-}$ shifted by $\chi_{p}(\sl,\cdot)-\chi_{l}(\sl,\cdot)$. The same for (\ref{np1}).
Next,
$$
\bigwedge\nolimits^{\!\bullet}{\left(T_l \X^{\vee}\right)}=\prod\limits_{{\delta \in\mathrm{char}_{\bA}(T_l \X)}} \left(1-\aa^{\delta}\right)\in K_{\bA}(pt).
$$
We conclude that 
$$
\tilde{Z}^{''}_{p,r} = \sum\limits_{l\in \X^{\bT}} \, \dfrac{f_l}{\prod\limits_{{\delta \in\mathrm{char}_{\bA}(T_l \X)}} \left(1-\aa^{\delta}\right)}
$$
where $f_l$ is a Laurent polynomial in equivariant parameters $\aa$ whose Newton polygon is contained in the Newton polygon of the denominator $$\prod\limits_{{\delta \in\mathrm{char}_{\bA}(T_l \X)}} (1-\aa^{\delta})$$ after the shift by the character
$
\chi_{p}(\sl,\cdot)-\chi_{r}(\sl',\cdot)+\chi_{l}(\sl'-\sl,\cdot).
$
We can rewrite this shift as:
$$
\chi_{p}(\sl,\cdot)-\chi_{r}(\sl,\cdot)-\chi_{r}(\sl'-\sl,\cdot)+\chi_{l}(\sl'-\sl,\cdot).
$$
Thus, at arbitrary infinity $\aa \to \infty$ of the torus $\bA$ the terms of $\tilde{Z}^{''}_{p,r}$ grow not faster than 
$$
\aa^{\chi_{p}(\sl,\cdot)-\chi_{r}(\sl,\cdot)-\chi_{r}(\sl'-\sl,\cdot)+\chi_{l}(\sl'-\sl,\cdot)} 
$$
By assumption of the theorem $\sl'-\sl$ is a small slope. It means that  that the degrees of $\aa$-monomials appearing in the Laurent polynomial $\tilde{Z}^{''}_{p,r}$ are located in some small neighborhood of $\chi_{p}(\sl,\cdot)-\chi_{r}(\sl,\cdot)$. The only possibility is if $\tilde{Z}_{p,r}$ is itself a monomial in $\aa$
$$
\tilde{Z}^{''}_{p,r}=\aa^{\chi_{p}(\sl,\cdot)-\chi_{r}(\sl,\cdot)} Z^{'}_{p,r} 
$$
for some $Z^{'}_{p,r}\in \matQ(\zz,\hbar)$. Finally, if $Z^{'}_{p,r}\neq 0$ then $\tilde{Z}^{''}_{p,r}$ is a Laurent polynomial only if its weight is integral $\chi_{p}(\sl,\cdot)-\chi_{r}(\sl,\cdot) \in \bA^{\wedge}$.
\end{proof}

\begin{Corollary} \label{cor2}
{\it Let $\sl$ be such that it belongs to exactly one hyperplane of $\mathsf{Wall}(\X)$ then
$$
\tilde{Z}^{'}_{p,r}\neq 0 \ \ \Longrightarrow \ \ \chi_{p}-\chi_{r} = [C] \otimes v \in H_{2}(\X,\matZ)_{\textrm{eff}} \otimes \bA^{\wedge}_{>} 
$$}
\end{Corollary}
\begin{proof}
In this case the values $\langle \chi_p-\chi_r, \sl \rangle$ is the same for all $\sl$ on the wall. Its only possible if 
$$
\chi_{p}-\chi_{r} = [C] \otimes v 
$$
and the equation of the wall is $\langle [C],\sl \rangle=n $ for some $n\in \matZ$. The result follows from Proposition $\ref{dprop}$.
\end{proof}

\subsection{}
Let $\mathscr{U}_{0} \subset H^{2}(\X,\matR)$ denote a small analytic neighborhood of $0\in H^{2}(\X,\matR)$.\footnote{By small we mean that 
$$
\mathscr{U}_{0} \cap (\textsf{Wall}(\X)\setminus \textsf{Wall}_{0}(\X)) =\emptyset
$$}
Let $\textsf{Wall}_{0}(\X) \subset \textsf{Wall}(\X)$ denote the subset of hyperplanes passing through $0$ and
$$
\mathscr{U}_{0}\setminus  \textsf{Wall}_{0}(\X) = \coprod \, \fD(\X)
$$
be the decomposition to connected components (chambers). We denote by $\fD_{+}(\X)$ the chamber which contains ample line bundles, 
and by $\fD_{-}(\X)=-\fD_{+}(\X)$. 
The elements of $\mathscr{U}_{0}$ are called {\it small slopes}. The elements of $\fD_{+}(\X)$ (respectively from $\fD_{-}(\X)$) are called {\it small ample slopes} (respectively {\it small anti-ample}).

\subsection{} 

We denote by
$$
\fD_{\pm}(\X^{!})=d\kappa(\pm \fC) \cap \mathscr{U}_{0} \subset H^{2}(\X^{!},\matR) 
$$
small slopes for $\X^{!}$. By Theorem \ref{thm2}, the resonances of $\X$ are the same as walls of $\X^{!}$. Thus $\fD_{+}(\X^{!})$ and $\fD_{-}(\X^{!})$ are actually small ample and anti-ample slopes of $\X^{!}$.

If $\sl$ is not regular then by Theorem \ref{thm2}
we can view it as element of $\mathsf{Res}(\X^{!})$ and thus we have a non-trivial subvariety $\Y_{\sl} \subset \X^{!}$.  
Finally we denote $\fD_{\pm}(\Y_{\sl}) = i^{*} (\fD_{\pm}(\X^{!}))$ where $i^{*}:H^{2}(\X,\matR)\to H^{2}(\Y_{\sl},\matR)$ is induced by inclusion. 

\vspace{2mm}

\noindent In the notations of Section 3.2 of \cite{KononovSmirnov1}  we have:
\begin{Theorem} \label{factheorem}
{ \it Let $\sl\in \mathsf{Walls}(\X)$ and $\varepsilon \in \fD_{+}(\X)$ be a small ample (or anti-ample $\varepsilon \in \fD_{-}(\X)$ ) slope of $\X$, such that $\sl'=\sl+\varepsilon$ is a regular slope. Then, the matrix $\tilde{Z}'$ in Theorem \ref{factortheorem} has the form
$$
\tilde{Z}'=\hh\, \tilde{Z}\, \hh^{-1}
$$
where $\tilde{Z}$ is the matrix of K-theoretic stable envelopes  of $\,\Y_{s}$ with small ample slope:
$$
\tilde{Z}_{r,p}=\dfrac{\left.\Stab^{\fD_{+}(\Y_{\sl}),\Y_{\sl},{K}}_{\sigma^{!}}(p^{!})\right|_{r^{!}}}{\left.\Stab^{\fD_{+}(\Y_{\sl}),\Y_{\sl},{K}}_{\sigma^{!}}(p^{!})\right|_{p^{!}}}
$$
(respectively, 
$$
\tilde{Z}_{r,p}=\dfrac{\left.\Stab^{\fD_{-}(\Y_{\sl}),\Y_{\sl},{K}}_{\sigma^{!}}(p^{!})\right|_{r^{!}}}{\left.\Stab^{\fD_{-}(\Y_{\sl}),\Y_{\sl},{K}}_{\sigma^{!}}(p^{!})\right|_{p^{!}}}
$$
with small anti-ample slopes), and $\hh$ denotes a diagonal matrix in powers of $\hbar$:
\be \label{indexpart}
\hh:=\left.\mathrm{diag}\Big((-1)^{\gamma_p(\sl)} \hbar^{m_{p}(\sl)/2}  \Big)\right|_{p \in \X^{\bT}} , 
\ee
with $\gamma_{p}(\sl)=\mathrm{rk}(\ind_{p}-\ind^{\nn_{\sl}}_{p})$}
\end{Theorem}

\begin{proof}
Assume first that $\varepsilon \in \fD_{+}(\X)$. By factorization Theorem \ref{factortheorem} we have 
$$
\aa^{-\chi(\sl,\cdot) } \, \lim\limits_{q\to0}\tilde{T}^{\X} (\zz q^{\sl},\aa)\, \aa^{\chi(\sl,\cdot) }=\aa^{-\chi(\sl,\cdot) } \tilde{Z}^{''} \tilde{A}^{[\sl'],\X}  \aa^{\chi(\sl,\cdot)}=\tilde{Z}^{'} \aa^{-\chi(\sl,\cdot) } \tilde{A}^{[\sl'],\X}  \aa^{\chi(\sl,\cdot)}
$$
Note that for $p\succ r$ by Proposition \ref{dprop} we have 
$$
\chi_{p}(\sl'-\sl,\fC)-\chi_{r}(\sl'-\sl,\fC)>0
$$
The ``window'' condition for the $K$-theoretic stable envelopes, which bounds the $\aa$-degrees of matrix elements of $A^{[\sl'],\X}$, then implies that
$$
\lim\limits_{\aa\to 0_{\fC}} \left( \aa^{-\chi(\sl,\cdot) } A^{[\sl'],\X}  \aa^{\chi(\sl,\cdot)}\right)=\mathrm{Id}.
$$
where we denote
$
\lim\limits_{\aa\to 0_{\fC}} f(\aa) = \lim\limits_{z\to 0} f(\sigma(z))
$
for a cocharacter $\sigma : \matC^{\times} \rightarrow \bA$ from the chamber $\fC$. Thus, we have:
$$
\tilde{Z}^{'}=\lim\limits_{\aa\to 0_{\fC}} \left(\aa^{-\chi(\sl,\cdot) } \, \lim\limits_{q\to0}\tilde{T}^{\X}(\aa,\zz q^{\sl})\,  \aa^{\chi(\sl,\cdot)}\right).
$$
Now, we change the perspective - we consider the last limit from the point of view of $\X^{!}$ using (\ref{tmir}). We change the roles of parameters $\aa\leftrightarrow\zz$ using the isomorphism (\ref{kap}). 
From $\X^{!}$-standpoint the last limit has the form: 
$$
\tilde{Z}^{'}=\lim\limits_{\zz\to 0_{\fD_{+}(\X^{!})}} \left(\zz^{-\chi(\sl,\cdot) } \, \lim\limits_{q\to0}\tilde{T}^{\X^{!}} (\aa q^{\sl},\zz)\, \zz^{\chi(\sl,\cdot)}\right).
$$
({\it Important}: now $\aa$ denotes the equivariant and $\zz$ the K\"ahler parameters of $\X^{!}$). 
The proof follows from Theorem 2 in \cite{KononovSmirnov1}. 

The proof for $\varepsilon \in  \fD_{-}(\X)$ is the same with $\fC, \fD_{+}(\X^{!}),  \fD_{+}(\Y_{\sl})$ replaced by
$-\fC, \fD_{-}(\X^{!}),  \fD_{-}(\Y_{\sl})$ respectively.
\end{proof}
Now it becomes clear why operators in the factorization theorem (\ref{factortheorem}) are in that specific order. The reason is that there is natural identification of fixed points, while supposedly there are no natural isomorphisms between $K(X)$ and $K(X^!)$.

\section{K-theoretic duality interfaces} 
\subsection{} 

For $\sl=0$ by Theorems \ref{factortheorem} and 
we have
\be \label{factoriz}
\tilde{T}^{{K}}=\lim\limits_{q\to 0} \, \tilde{T}(\zz,\aa)= \tilde{Z}\, \tilde{A}
\ee
with 
$$
\tilde{A}_{p,r} = \dfrac{\left.\Stab^{[\pm \varepsilon],\X,{K}}_{\sigma}(p)\right|_{r}}{\left.\Stab^{[\pm \varepsilon],\X,{K}}_{\sigma}(r)\right|_{r}}, \ \ \  \tilde{Z}_{p,r} = \dfrac{\left.\Stab^{[\pm \varepsilon],\X^{!},{K}}_{\sigma^{!}}(r^{!})\right|_{p^{!}}}{\left.\Stab^{[\pm \varepsilon],\X^{!},{K}}_{\sigma^{!}}(p^{!})\right|_{p^{!}}}
$$
where $\pm \varepsilon$ denote small ample or anti-ample slopes.  
It is natural to consider the following matrix
\be \label{mmat}
M^{\pm}_{p,r}:= \left.\Stab^{[\pm \varepsilon],\X^{!},{K} }_{\sigma^{!}}(p^{!})\right|_{p^!} \,\tilde{T}^{{K}}_{p,r}\, \left.\Stab_{\sigma}^{[\pm \varepsilon],\X,{K}}(r)\right|_{r}.
\ee
with elements $M^{\pm}_{p,r} \in K_{\bT\times\bT^{!}}(pt)$. 
\begin{Proposition} \label{corrms}
{\it Let us consider a class $\mf^{\pm}_{0} \in K_{\bT\times \bT^{!}}(\X\times \X^{!})$ defined by
$$
\mf^{\pm}_{0}=\sum\limits_{p \in \mathsf{FP}}\,  \Stab_{\sigma}^{[\pm \varepsilon],\X,{K}}(p)\, \boxtimes\,
\Stab^{[\pm \varepsilon],\X^{!},{K} }_{\sigma^{!}}(p^{!}),
$$
where $\mathsf{FP}$ is the common set of torus fixed points (\ref{commonfp}), and $\boxtimes$ denotes the tensor product of K-theory classes of $\X\times \X^{!}$ pulled back from from the factors via the canonical projections then, \begin{itemize}
    \item These classes coincide, i.e,  $\mf_{0}:=\mf^{+}_{0}=\mf^{-}_{0}$.
    \item The matrix (\ref{commonfp}) is the fixed point components of $\mf_{0}$:
    $$
    (r,p^{!}) \in (\X \times \X^{!})^{\bT\times\bT^{!}} \ \ \Rightarrow \ \ \left.\mf_{0}\right|_{(r,p^{!})}=M^{+}_{p,r}=M^{-}_{p,r}.
   $$
\end{itemize}
}
\end{Proposition} 
\begin{proof}
Clear from Theorems  \ref{factortheorem} and \ref{factheorem}. 
\end{proof}

\subsection{} 
Let us denote $K_{\bT}=K_{\bT}(pt)_\mathrm{loc}$, $K_{\bT^{!}}=K_{\bT^{!}}(pt)_\mathrm{loc}$.

Let $\pr_{\X}$ and $\pr_{\X^{!}}$ be projections (\ref{canprog}).
With our assumption on the fixed points, we can define push-forward maps using equivariant localization:
$$
K_{\bT}(\X)_\mathrm{loc}\otimes K_{\bT^{!}} \xleftarrow{\ \ \pr_{\X,*} \ \ } K_{\bT\times \bT^{!}} (\X\times \X^{!}) \xrightarrow{\ \ \pr_{\X^{!},*} \ \ } K_{\bT^{!}}(\X^{!})_\mathrm{loc}\otimes  K_{\bT}. 
$$
Thus we can define maps of $K$-theories:
$$
\ms_{\mf_0}: K_{\bT}(\X) \rightarrow K_{\bT^{!}}(\X^{!})_\mathrm{loc}\otimes K_{\bT}, \ \ \     \ms^{t}_{\mf_0}: K_{\bT^{!}}(\X^{!}) \rightarrow K_{\bT}(\X)_\mathrm{loc}\otimes K_{\bT^{!}}
$$
defined by 
$$
\ms_{\mf_{0}}: c \mapsto    \pr_{\X^{!},*}(\mf_{0}\otimes  \pr_{\X}^{*}(c)), \ \ \ \ms_{\mf_0}^{t}: c \mapsto    \pr_{\X,*}(\mf_{0}\otimes  \pr_{\X^{!}}^{*}(c)).
$$
\begin{Proposition}
{\it The correspondences $\ms_{\mf_{0}}$ and $\ms^{t}_{\mf_0}$ map the stable envelope classes to the stable envelope classes:
$$
\ms_{\mf_{0}}\Big( \Stab^{[\pm \varepsilon],\X,{K}}_{\sigma}(p) \Big)=  \Stab^{[\mp \varepsilon],\X^{!},{K}}_{-\sigma^{!}}(p^{!}), 
$$
$$
\ms_{\mf_{0}}^{t}\Big( \Stab^{[\pm \varepsilon],\X^{!},{K}}_{\sigma^{!}}(p^{!}) \Big)=
\, \Stab^{[\mp \varepsilon],\X,{K}}_{-\sigma}(p).
$$}
\end{Proposition}
\begin{proof}
The proof follows from Proposition \ref{corrms} and the following orthogonality of stable envelopes with respect to the Euler characteristic $\chi_{\X}$
$$
\chi_{\X}\Big(\Stab^{[ \sl],\X,{K}}_{\sigma}(p) \otimes \Stab^{[-\sl],\X,{K}}_{-\sigma}(r)\Big)=\delta_{p,r}
$$
which holds for all $\sl$ and $\sigma$, see Section 1.9.16 in \cite{pcmilect}.
\end{proof}

\subsection{} 
For $\sl\neq 0$  the situation differs only by appearance of $\aa^{\chi(\sl,\cdot)}$ in Theorem \ref{factortheorem}. This weight 
$\chi_p(\sl,\cdot) \in H^{2}(\X,\matQ)$ is not necessarily integral  and thus $\aa^{\chi_p(\sl,\cdot)} \not\in K_{\bT}(\X)$ in general. Therefore, we are forced to work with certain extensions of the $K$-theory rings.  In this section, $\widehat{K}$ denotes a formal extension of a ring $K$ by the elements $\aa^{\pm \chi_p(\sl,\cdot)}$.

We can consider the following matrix
$$
\begin{array}{ll}
M^{[\sl],\pm}_{p,r}=  \\   (-1)^{\gamma_p(\sl)} \hbar^{-m_p(\sl)/2} \aa^{\chi_{p}(\sl,\cdot)} H_p^{-1} \left.\Stab^{\fD_{\pm}(\Y_{\sl}),\Y_{\sl},{K} }_{\sigma^{!}}(p^{!})\right|_{p^{!}} \,\tilde{T}^{[\sl],{K}}_{p,r}\, \left.\Stab_{\sigma}^{[\sl\pm \varepsilon],\X,{K}}(r)\right|_{r} 
\end{array}
$$
where $m_p(\sl)$ and $\gamma_p(\sl)$ are as in (\ref{indexpart}). By definition $M^{[\sl],\pm}_{p,r} \in \widehat{K}_{\bT\times\bT^{!}}(pt)$.

Similarly to our consideration in the previous subsection we conclude that the coefficients of this matrix glue to a $K$-theory class $\mf_{\sl}$ in $\hat{K}_{\bT\times \bT^{!}}(\X\times \X^{!}).$
$$
\mf_{\sl}^{\pm}=\sum\limits_{p \in \textsf{FP}}\, ( (-1)^{\gamma_p(\sl)} \hbar^{-m_{p}(\sl)/2}  \aa^{-\chi_{p}(\sl,\cdot)}  \Stab_{\sigma}^{[\sl\pm \varepsilon],\X,{K}}(p))\, \boxtimes\,
\Stab^{\fD_{\pm}(\Y_{\sl}),\Y_{\sl},{K} }_{\sigma^{!}}(p^{!})
$$
and arguing as above we obtain:
\begin{Theorem} 
{\it 

\vspace{2mm}
\indent
\begin{itemize} \label{stheor}
    \item 
    $$
    \mf_{\sl}:=\mf^{+}_{\sl}=\mf^{-}_{\sl}
    $$
    \item The matrix $M^{[\sl],+}_{p,r}=M^{[\sl],-}_{p,r}$ is the fixed point components of $\mf_{\sl}$:
    $$
    (r,p^{!}) \in (\X \times \Y)^{\bT\times\bT^{!}} \ \ \Rightarrow \ \ \left.\mf_{\sl}\right|_{(r,p^{!})}=M^{+}_{p,r}=M^{-}_{p,r}.
   $$
   \item The correspondences 
   $$
\ms_{\mf_\sl}: K_{\bT}(\X) \rightarrow \widehat{K}_{\bT^{!}}(\Y)_{loc}\otimes K_{\bT}, \ \ \     \ms^{t}_{\mf_0}: {K}_{\bT^{!}}(\Y) \rightarrow \widehat{K}(\X)_{loc}\otimes K_{\bT^{!}}
$$
defined by 
$$
\ms_{\mf_{\sl}}: c \mapsto    \pr_{\,\X^{!},*}(\mf_{\sl}\otimes  \pr_{\,\X}^{*}(c)), \ \ \ \ms_{\mf_0}^{t}: c \mapsto    \pr_{\,\X,*}(\mf_{\sl}\otimes  \pr_{\,\X^{!}}^{*}(c)),
$$
map the stable envelope classes to the twisted stable envelope classes
$$
\ms_{\mf_\sl}(\Stab^{[\sl\pm \varepsilon],\X,{K}}_{ \sigma}(p))=(-1)^{\gamma_p(\sl)} \aa^{-\chi_{p}(\sl,\cdot)} \hbar^{-m_{p}(\sl)/2} \, \Stab^{[\mp \varepsilon],\Y_{\sl},{K}}_{-\sigma^{!}}(p^!), 
$$
$$
\ms^{t}_{\mf_\sl}(\Stab^{[\pm \varepsilon],\Y_{\sl},{K}}_{ \sigma^{!}}(p^!))=(-1)^{\gamma_p(\sl)} \aa^{\chi_{p}(\sl,\cdot)} \hbar^{m_{p}(\sl)/2} \, \Stab^{[\sl\mp \varepsilon],\X,{K}}_{-\sigma}(p).
$$
\end{itemize}
}
\end{Theorem}

\section{Wall-crossing operators \label{wcsec}}

\subsection{}

Let us consider a slope $\sl \in H^{2}(\X,\matR)$. We can always choose small ample slope $\varepsilon \in \fD_{+}(\X)$ so that 

\begin{itemize}
\item $\sl\pm \varepsilon \not\in \textsf{Wall}(\X)$, i.e., both $\sl +\varepsilon$ and
$\sl-\varepsilon$ are regular.

\item If $\sl$ in not regular, then it is the only non-regular slope on the $\matR$-segment connecting points $\sl +\varepsilon$ and $\sl-\varepsilon$ in $H^{2}(\X,\matR)$:
$$
(\sl-\varepsilon,\sl+\varepsilon) \cap \textsf{Wall}(\X) = \left\{\begin{array}{l}\sl\\ \emptyset
\end{array}\right.
$$
\end{itemize}
In this situation, we define the following operator
\be \label{Rw}
\mathsf{R}^{\X}(\sl,\sigma)=(\Stab^{[\sl-\varepsilon],\X,{K}}_{\sigma})^{-1}  \circ \Stab^{[\sl+\varepsilon],\X,{K}}_{\sigma}  \in \mathrm{End}(K_{\bT}(\X^{\bA})_{loc}).
\ee
This operator describes the change of $K$-theoretic stable envelope corresponding to infinitesimal change of slope parameters from $\sl-\varepsilon$ to $\sl+\varepsilon$. Clearly, if $\sl$ is regular then $\mathsf{R}^{\X}(\sl,\sigma)=\textrm{Id}$.

\begin{Definition}
{\it 
If $\sl$ belongs to exactly one hyperplane of $\mathsf{Wall}(\X)$ then $\mathsf{R}^{\X}(\sl,\sigma)$ is called 
wall $R$-matrix. }
\end{Definition}

\begin{Remark}
Another distinguished operator is $\mathsf{R}^{\X}(0,\sigma)$. It describes the change of the $K$-theoretic stable envelopes from small anti-ample to small ample slopes.  
\end{Remark}

\begin{Remark} \label{remarkfactor}
Clearly, in general $\mathsf{R}^{\X}(\sl,\sigma)$,  is a product of several wall $R$-matrices. For instance
$$
\mathsf{R}^{\X}(0,\sigma)= \prod\, \mathsf{R}^{\X}(\sl,\sigma)
$$
where the product is over the wall $R$-matrices corresponding to the walls passing through $0\in H^{2}(\X,\matR)$.
\end{Remark}

\begin{Remark}
The wall $R$-matrices are important objects of geometric representation theory.  They  were investigated for $\X$ given by Springer resolutions in \cite{CSWallCrossing}. It was shown that the action of wall $R$-matrices generate the action of the affine Hecke algebra on $K_{\bT}(\X)$ is this case.

Interesting conjectures about these operators for
$\X=\mathrm{Hilb}^{n}(\matC^2)$ -- the Hilbert scheme of points in $\matC^2$ -- are discussed in \cite{NegGor} (see Section (\ref{hilbex}) below for this example).

In general, the wall $R$-matrices $\mathsf{R}^{\X}(\sl,\sigma)$ are solutions of the quantum Yang-Baxter equations, and thus can be used to construct actions of quantum groups on $K_{\bT}(\X^{\bA})$  \cite{OS}.
\end{Remark}

\begin{Proposition}[\cite{OS}]
{\it Let $\mathsf{R}^{\X}(\sl,\sigma)_{p,r}$ be the matrix elements of a wall $R$-matrix $\mathsf{R}^{\X}(\sl,\sigma)$ in the basis of the torus fixed points, then
$$
\mathsf{R}^{\X}(\sl,\sigma)_{p,r} \neq 0 \ \ \Longrightarrow \ \ \chi_p-\chi_r = [C] \otimes v \in H_2(X,\matZ)_{\textrm{eff}} \otimes \bA^{\wedge}_{>} 
$$}
\end{Proposition}
\begin{proof}
The argument repeat the proof of Corollary \ref{cor2} and Theorem \ref{factortheorem} with $\Gamma(p)$ replaced by $\Stab^{[\sl+\varepsilon],\X,{K}}_{\sigma}(p)$.
\end{proof}

\subsection{} 
The wall $R$-matrices  of dual varieties $\X$ and $\X^{!}$ are related.

\begin{Theorem}  \label{wallcrth}
{\it Let $\sl \in H^{2}(\X,\matR)$ and let $\Y_{\sl} \subset \X^{!}$ be the corresponding subvariety, then
$$
\mathsf{R}^{\X}(\sl,\sigma)= \aa^{\chi(\sl,\cdot)} \hh\, \mathsf{R}^{\Y_{\sl}}(0,-\sigma^{!})^{-1}\, \hh^{-1} \aa^{-\chi(\sl,\cdot)}. \ \ \ 
$$}
\end{Theorem}
\begin{proof}
The result follows immediately from the last item of Theorem \ref{stheor}. By the definition of 
$\mathsf{R}^{\X}(\sl,\sigma)$ we have
$$
\Stab^{[\sl+\varepsilon],\X,{K}}_{\sigma}= \Stab^{[\sl-\varepsilon],\X,{K}}_{\sigma}\, \mathsf{R}^{\X}(\sl,\sigma)
$$
applying $\ms_{\mf_\sl}$ as in Theorem \ref{stheor} gives
$$
 \Stab^{[-\varepsilon],\Y_{\sl},{K}}_{-\sigma^{!}} \aa^{-\chi(\sl,\cdot)} \hh^{-1} =  \Stab^{[+\varepsilon],\Y_{\sl},{K}}_{-\sigma^{!}}\, \aa^{-\chi(\sl,\cdot)} \hh^{-1}\,\mathsf{R}^{\X}(\sl,\sigma)\,
$$
which means
$$
\Stab^{[+\varepsilon],\Y_{\sl},{K}}_{-\sigma^{!}}=\Stab^{[-\varepsilon],\Y_{\sl},{K}}_{-\sigma^{!}}\,\aa^{-\chi(\sl,\cdot)} \hh^{-1}\,\mathsf{R}^{\X}(\sl,\sigma)^{-1} \aa^{\chi(\sl,\cdot)} \hh
$$
thus
$$
\aa^{-\chi(\sl,\cdot)} \hh^{-1}\,\mathsf{R}^{\X}(\sl,\sigma)^{-1} \aa^{\chi(\sl,\cdot)} \hh=\mathsf{R}^{\Y_{\sl}}(0,-\sigma^{!})
$$
the inverse of this identity gives the result. 
\end{proof}

The following elementary corollary gives new vanishing conditions for the matrix elements of $\mathsf{R}^{\X}(\sl,\sigma)$:
\begin{Corollary}
{\it If $p^{!}$ and $r^{!}$ belong to different components of $\Y_{\sl}$ then
$$
\mathsf{R}^{\X}(\sl,\sigma)_{p,r}=0.
$$}
\end{Corollary}

\begin{Remark}
We note that by Theorem \ref{wallcrth},  a wall $R$-matrix $\mathsf{R}^{\X}(\sl,\sigma)$ factors into a product of wall $R$-matrices $\mathsf{R}^{\Y_{\sl}}$ associated with $\Y_\sl\subset \X^{!}$, see Remark~$\ref{remarkfactor}$. This factorization, in turn, can be applied to each of the new factors $\mathsf{R}^{\Y_{\sl}}$ and so on.  This recursion leads to a factorization of wall a $R$-matrix into certain ``elementary matrices'', which can not be further factorized. 
\end{Remark}

\section{Application: Hilbert scheme $\X=\mathrm{Hilb}^{n}(\matC^2)$ \label{hilbex}}

\subsection{} 
For a natural number $n$ let  $\X=\textrm{Hilb}^{n}(\matC^2)$ denote the Hilbert scheme of $n$ points in $\matC^2$. This space satisfies all the conditions discussed in Section \ref{datasec}. The Hilbert scheme is known to be {\it self-dual} in the sense that there is an isomorphism
$
\X^{!} \cong \X
$ \cite{AOprep}.


The Hilbert scheme $\X$ is isomorphic to a Nakajima variety associated to the Jordan quiver with dimension vector $(n)$ and framing vector $(1)$ see Fig.\ref{jord}.   We refer to \cite{NakajimaLectures1} for a beautiful introduction into geometry of $\X$.
The elliptic stable envelope classes for $\X$ were computed in \cite{SmirnovElliptic}.

\begin{figure}[H]
	\centering
	\includegraphics[width=3cm]{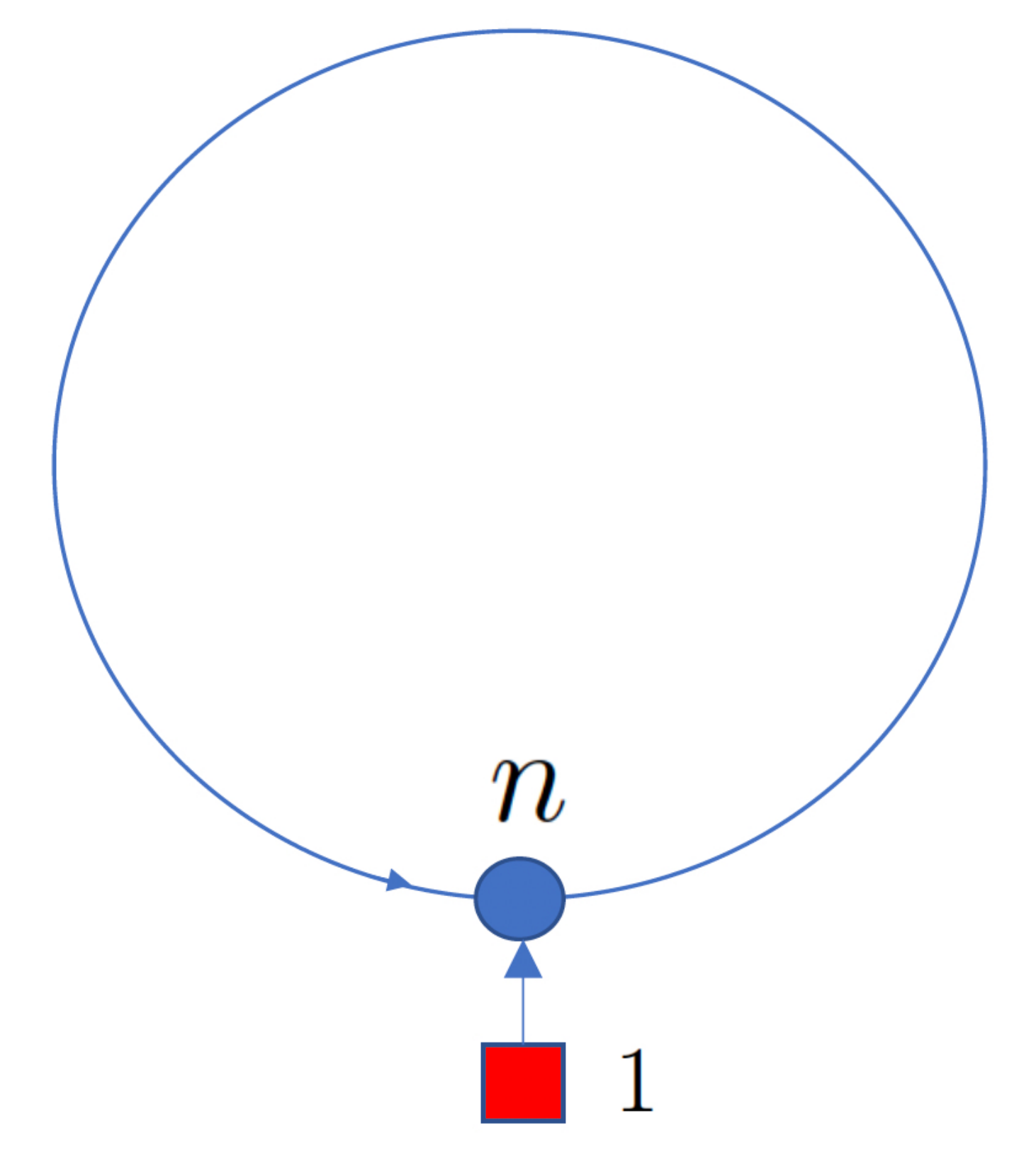}
	\caption{\label{jord} The quiver defining the Hilbert scheme $\X=\textrm{Hilb}^{n}(\matC^2)$.}
\end{figure}
\noindent

\subsection{}

In this case $\Pic(\X)\cong \matZ$ is generated by $\mathscr{O}(1)$. We identify $H^{2}(\X,\matR)=\matR$ so that the integer points correspond to $c_1(\mathscr{O}(m))$, $m\in \matZ$.

\begin{Proposition}[\cite{KononovSmirnov1}]
{\it
\vspace{1mm}

\indent
\begin{itemize} 
    \item Under the identification $H^{2}(\X,\matR)=\matR$ the walls of $\X$ are located at the following rational points
$$
\mathsf{Walls}(\X) = \left\{ \dfrac{a}{b} \in \matQ: |b|\leq n  \right\}  \cong \matR
$$ 
\item For a slope $\sl = \frac{a}{b}$ the subvariety $\Y_{\sl}\subset \X^{!}$ has the following form:
$$
\Y_{\sl}= \coprod_{{n_0,n_1,\dots,n_{b-1}} \atop {n_0+\dots + n_{b-1}=n}} \X(n_0,\dots, n_{b-1})   
$$
Its connected components $\X(n_0,\dots, n_{b-1})$ are isomorphic to the Nakajima varieties associated to the cyclic quiver with $b$ vertices, the dimension vector $(n_0,\dots,n_{b-1})$ and the framing dimension vector $(1,0,\dots, 0)$, see Fig. \ref{cycl}. 
\end{itemize}}
\end{Proposition}

\begin{figure}[H]
	\centering
	\includegraphics[width=4cm]{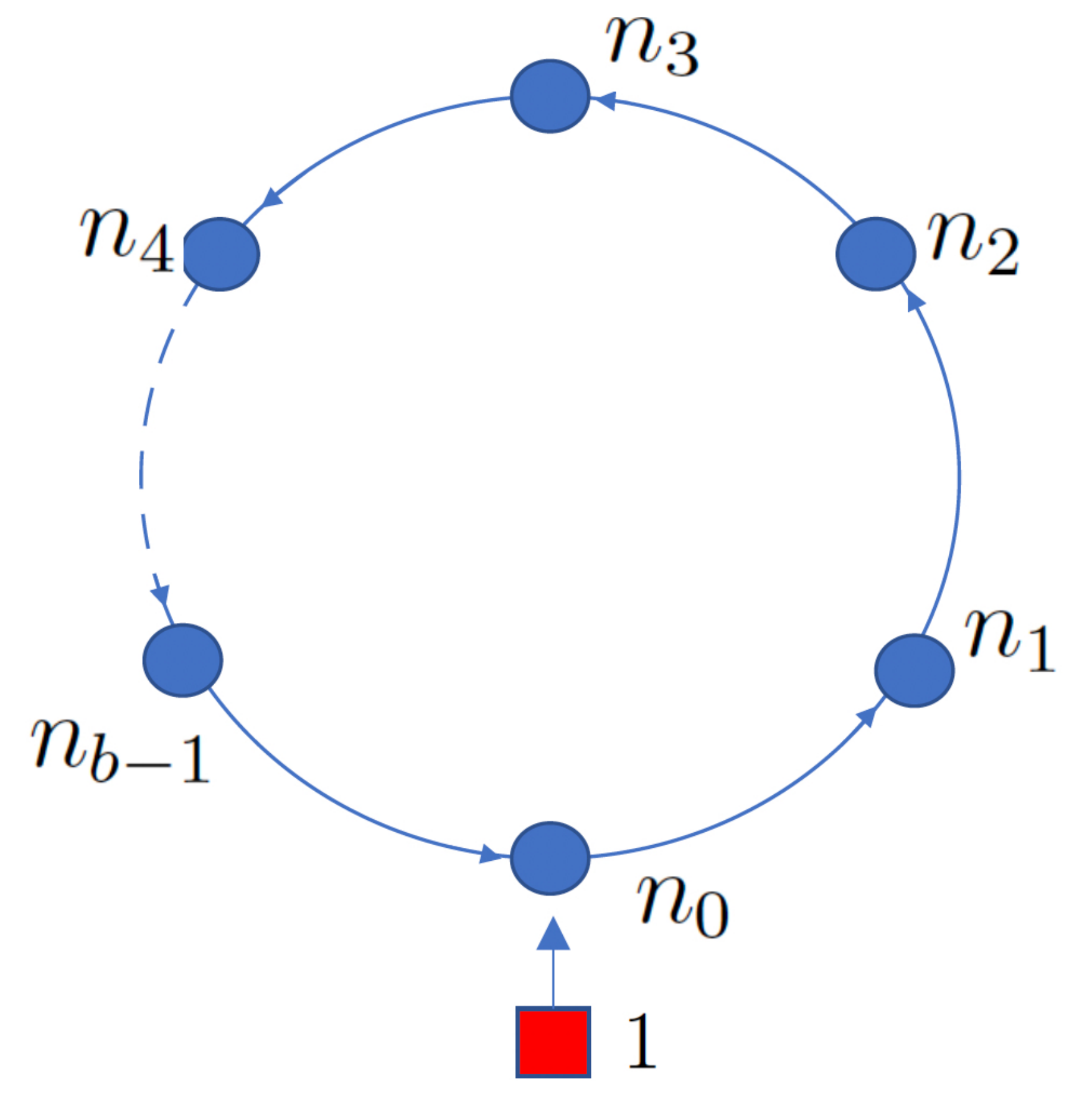}
	\caption{\label{cycl} The quiver defining the Nakajima variety $\X(n_0,\dots,n_{b-1})$.}
\end{figure}
\noindent

\subsection{} 

From the representation theoretic standpoint, the space 
\be \label{fockmod}
\mathsf{Fock}:= \bigoplus_{n=0}^{\infty} K_{\bT}(\Y_{\sl})  \ \ =\bigoplus_{n_0,\dots, n_{b-1}=0}^{\infty} K_{\bT}(\X(n_0,\dots, n_{b-1}))
\ee
is equipped with a natural action of the quantum affine algebra $\mathscr{U}_{\hbar}(\widehat{\frak{gl}}_{b})$ \cite{Nak1}. The $K$-theoretic stable envelope bases of $\X(n_0,\dots, n_{b-1})$ with small ample and anti-ample slopes correspond to the so-called global {\it standard} and {\it co-standard bases} of the Fock module \cite{Neg}.
Theorem \ref{stheor} then gives: 
\begin{Theorem}
{\it Let $\sl=\frac{a}{b} \in \mathsf{Walls}(\X)$, then under isomorphism (\ref{fockmod}) the $K$-theoretic duality interface $\ms_{\mf_\sl}$ maps the standard
and co-standard bases of the Fock $\mathscr{U}_{\hbar}(\widehat{\frak{gl}}_{b})$-module to the stable bases of $\X$:
$$
\begin{array}{cc}
\ms^{t}_{\mf_\sl}: \textrm{standard basis of $\mathscr{U}_{\hbar}(\widehat{\frak{gl}}_{b})$ $\mathsf{Fock}$ module } \longrightarrow  \hh_p \,\Stab^{[\sl+\varepsilon],\X,{K}}_{\sigma}(p)\\
\\
\ms^{t}_{\mf_\sl}: \textrm{co-standard basis of $\mathscr{U}_{\hbar}(\widehat{\frak{gl}}_{b})$ $\mathsf{Fock}$ module } \longrightarrow  \hh_p \, \Stab^{[\sl-\varepsilon],\X,{K}}_{\sigma}(p)\
\end{array}
$$
where $\varepsilon$ denotes small ample slope and
$\hh_p$ is the monomial
\be \label{diamh}
\hh_p=(-1)^{\gamma_p} \hbar^{-m_p(\sl)/2} \aa^{\chi_{p}(\sl,\cdot)}.
\ee
}
\end{Theorem}

\noindent
This result proves the main conjecture of  \cite{NegGor}. 

\begin{Remark}
The prefactor $\hh_p$ is computed in Section  4 of our previous paper \cite{KononovSmirnov1}. It coincides with the renormalization of stable basis suggested in Section 4.4 of~\cite{NegGor}.
\end{Remark}

\subsection{} 
Under the isomorphism of $\mathscr{U}_{\hbar}(\widehat{\frak{gl}}_{b})$-modules (\ref{fockmod}) the operator
$\mathsf{R}^{\Y_{\sl}}(0)^{-1}$
is the transition matrix from the standard to co-standard bases of the Fock module. Theorem \ref{wallcrth} thus gives:

\begin{Theorem} 
{\it The wall $R$-matrix $\mathsf{R}^{\X}(\sl)$ for $\X=\mathrm{Hilb}^{n}(\matC^{2})$ and $\sl = \frac{a}{b}$ coincides, up to a conjugation by the diagonal matrix (\ref{diamh}), with the matrix of transition from the standard to the co-standard basis in the Fock module of $\mathscr{U}_{\hbar}(\widehat{\frak{gl}}_{b})$.}
\end{Theorem}

\begin{Remark}
The above results imply that for a rational number $\sl=\frac{a}{b}$ there exists an action of
$\mathscr{U}_{\hbar}(\widehat{\frak{gl}}_{b})$ on
$$
\bigoplus\limits_{n=0}^{\infty} \, K_{\bT}(\mathrm{Hilb}^{n}(\matC^2))
$$
thus confirming the prediction of \cite{NegGor}. This  action may also be studied by non-geometric representation theoretic techniques see \cite{BerGon} for recent advances.
\end{Remark}

\bibliographystyle{abbrv}
\bibliography{bib}

\newpage

\vspace{12 mm}

\noindent
Yakov Kononov\\
Department of Mathematics,\\
Columbia University,\\
New York, NY 10027, USA\\
ya.kononoff@gmail.com

\vspace{3 mm}

\noindent
Andrey Smirnov\\
Department of Mathematics,\\
University of North Carolina at Chapel Hill,\\
Chapel Hill, NC 27599-3250, USA;\\
Steklov Mathematical Institute\\
of Russian Academy of Sciences,\\
Gubkina str. 8, Moscow, 119991, Russia.\\
asmirnov@email.unc.edu\\

\end{document}